\documentclass[12pt,a4paper]{article}
\usepackage[T1]{fontenc}
\usepackage[utf8]{inputenc}
\usepackage[british]{babel}
\usepackage{amsmath,amssymb,amsthm}
\usepackage{tensor}
\usepackage{xcolor}
\usepackage[all]{xy}
\usepackage[noadjust]{cite}
\usepackage{hyperref}

\let\conj=\induced

\newcommand{\N}{\mathbb{N}} 
\newcommand{\Z}{\mathbb{Z}}

\DeclareMathOperator{\centre}{Z}
\DeclareMathOperator{\Cent}{C}
\DeclareMathOperator{\Norm}{N}
\DeclareMathOperator{\Soc}{Soc}

\DeclareMathOperator{\Aut}{Aut}

\DeclareMathOperator{\Fix}{Fix}

\DeclareMathOperator{\id}{id}

\newtheorem{teo}{Theorem}[section]
\newtheorem{prop}[teo]{Proposition}
\newtheorem{lemma}[teo]{Lemma}

\newtheorem{athm}{Theorem}

\theoremstyle{definition}

\theoremstyle{remark}

\title{Analogues of Gr\"un's lemma and Baer's theorem for skew left braces}
\author{A. Ballester-Bolinches\thanks{Departament de Matem\`atiques, Universitat de Val\`encia, Av.\ Vicent Andr\'es Estell\'es, 19, 46100 Burjassot, Val\`encia, Spain, \texttt{\href{mailto:Adolfo.Ballester@uv.es}{Adolfo.Ballester@uv.es}}, \texttt{\href{mailto:Ramon.Esteban@uv.es}{Ramon.Esteban@uv.es}}, \texttt{\href{mailto:Pedro.A.Perez@uv.es}{Pedro.A.Perez@uv.es}}, ORCID \href{https://orcid.org/0000-0002-2051-9075}{0000-0002-2051-9075}, \href{https://orcid.org/0000-0002-2321-8139}{0000-0002-2321-8139},  \href{https://orcid.org/0009-0009-7082-9002}{0009-0009-7082-9002}}\and R. Esteban-Romero\addtocounter{footnote}{-1}\footnotemark{} \and L. A. Kurdachenko\thanks{Department of Algebra and Geometry, Oles Honchar Dnipro National University,
Dnipro 49010, Ukraine, \texttt{\href{mailto:lkurdachenko@gmail.com}{lkurdachenko@gmail.com}}, ORCID \href{https://orcid.org/0000-0002-6368-7319}{0000-0002-6368-7319}.}\ \thanks{Part of the research of this author has been carried out in the Departament de Matemàtiques, Universitat de València; Av.\ Vicent Andr\'es Estell\'es, 19, 46100 Burjassot, València, Spain.} \and P. P\'erez-Altarriba\addtocounter{footnote}{-3}\footnotemark{}}

\begin{document}
\maketitle
\begin{abstract}
  We prove in this paper some analogues of the well-known group-theoretical Gr\"un's lemma, stating that in a perfect group the first and the second centre coincide, and Baer's theorem, stating that if the quotient by the $n$th centre of a group is finite, then so is the $(n+1)$th term of the lower central series, in the scope of infinite skew left braces. These results represent significant improvements over previous work. The trifactorised group associated with a skew left brace will be crucial for our proofs.


\end{abstract}
\section{Introduction}
In group theory, the boundaries of nilpotency and central stability in group extensions are better understood by means of Gr\"un's lemma and Baer's theorem. Gr\"un's lemma \cite[Satz~4]{Grun36} establishes that perfect groups exhibit an immediate central stabilisation (namely, $\operatorname{Z}_2(G)=\operatorname{Z}(G)$), while Baer's theorem \cite[\S6, Endlichkeitssatz]{Baer52-Endlichkeitskriterien} exhibits a dual perspective as the fact that a quotient $G/{\operatorname{Z}_n(G)}$ by a term of the upper central series of a group is finite forces the finiteness of the corresponding term $\gamma_{n+1}(G)$ of the lower central series. Baer's theorem for $n=1$ is also known as Schur's theorem. Taken together, Gr\"un's lemma and Baer's theorem show two complementary aspects of the relationship between centrality and commutator behaviour: the first one emphasises structural collapse in the presence of maximal non-abelianity, while the second one shows how finiteness conditions imposed on the quotients by the terms of the ascending central series propagate into the descending central series. They are essential tools in the study of nilpotency, central extensions, and generalised finiteness conditions in groups.


An interesting algebraic structure that can be regarded as an extension of groups and radical rings and that models combinatorial solutions of the Yang-Baxter equation of mathematical physics \cite{Drinfeld92,Baxter72,Yang67} is the one of skew left brace introduced in \cite{GuarnieriVendramin17,Rump07}. Skew left braces provide a comprehensive algebraic framework to construct and analyse non-degenerate solutions of the Yang-Baxter equation.
A \emph{skew left brace} (or, in short, a \emph{skew brace}) $(B, {+}, {\cdot})$ consists of a set $B$ with two binary operations $+$ and~$\cdot$ such that $(B, {+})$ and $(B, {\cdot})$ are groups linked by a distributive-like law $a(b+c)=ab-a+ac$ for $a$, $b$, $c\in B$. When the operations on $B$ are clear from the context, we will refer to a skew brace as $B$ instead of $(B, {+}, {\cdot})$. It is abundantly clear that the more we understand skew  braces, the more we understand combinatorial solutions of the Yang-Baxter equation.

This paper springs out from reading the papers \cite{Tsang26-jmathsocjapan} of Tsang and \cite{JespersKubatVanAntwerpenVendramin21} of Jespers, Kubat, Van Antwerpen, and Vendramin. The main result of \cite{Tsang26-jmathsocjapan} can be regarded as a generalisation of Gr\"un's lemma to two-sided skew braces, while \cite[Theorem~5.4]{JespersKubatVanAntwerpenVendramin21} can be regarded as a skew brace version of Schur's theorem.
In this paper, we present significant improvements of the above results. We expect, as in the case of groups, that these results will help us understand the relation between centrality and commutator behaviour in skew braces and can be essential tools in the study of nilpotency, central extensions, and generalised finiteness conditions in skew braces. The trifactorised groups associated with skew braces studied in \cite{BallesterEsteban22,BallesterEstebanPerezAPerezC26-commmathstat-categoriesskewleftbraces} are an essential tool in the proofs of our results. This confirms the strength of the techniques based on trifactorised groups to understand the structure of skew braces.

\section{Terminology and statement of results}

In order to state the results for skew braces, we must introduce some terminology and notation. In this section, $(B,+,\cdot)$, or simply $B$, will denote a skew brace. 

An important operation in $B$ that relates both group operations is the \emph{star} operation defined by $a*b=-a+ab-b$ for $a$, $b\in B$. Indeed, both group operations coincide if, and only if, $a*b=0$ for all $a$, $b\in B$; in this case, $B$ is said to be a \emph{trivial skew brace}. Groups can be identified with trivial skew braces. Trivial skew braces with abelian additive group are called \emph{abelian skew braces}. The star operation can be regarded as a skew brace analogue of the commutator in group theory and is used to define nilpotency properties in skew braces.

Now we recall the definitions of  substructures of~$B$.
A \emph{subbrace} of $B$ is a subgroup of $(B, {+})$ which is also a subgroup of $(B, {\cdot})$.

Given two subsets $X$, $Y \subseteq B$, we write $\langle X \rangle_+$,  $[X,Y]_+$ for the subgroup generated by $X$ and the commutator of $X$ and $Y$ in $(B,+)$, respectively, and we write $\langle X\rangle_\bullet$, $[X,Y]_\bullet$ for the subgroup generated by $X$ and the commutator of $X$ and $Y$ in $(B,\cdot)$, respectively. The subgroup of $(B, +)$ generated by the set $\{x \ast y \mid x \in X, y \in Y\}$ is denoted by $X \ast Y$. 
 
A subbrace $L$ of $B$ is said to be a \emph{left ideal} (respectively, \emph{right ideal}) of $B$ if $B \ast L \subseteq L$ (respectively, $L \ast B \subseteq L$). A \emph{strong left ideal} $I$ of $B$ is a left ideal such that $(I,{+})$ is a normal subgroup of $(B,{+})$.  An \emph{ideal} of $B$ is a strong left ideal which is a right ideal of $B$. Note that every subgroup of a trivial skew brace is a left and right ideal, and $B\ast B$ is an ideal of $B$.

We define the \emph{left series} of $B$ by means of $B^1=B$ and $B^n=B\ast B^{n-1}$ for $n\geq 2$. It is well known that the terms of the left series are left ideals.

We know that $(B,{\cdot})$ acts on $(B,{+})$ via the \emph{lambda map} given by \[\lambda\colon (B,{\cdot})\longrightarrow\Aut(B,{+}),\quad\lambda(a)(b)=\lambda_a(b)=-a+ab,\quad a, b\in B.\] Furthermore, this action gives us some important substructures as the \emph{kernel of lambda}, the set of \emph{fixed points} of $B$, the \emph{socle} of $B$ and the \emph{centre} of $B$:
\begin{align*}
\ker\lambda&=\{a\in B\mid\lambda_a(b)=b,\;\forall b\in B\},\\
\Fix(B)&=\{b\in B\mid\lambda_a(b)=b,\;\forall a\in B\},\\
\Soc(B)&=\ker\lambda\cap \centre(B,{+}),\\
\centre(B)&=\Soc(B)\cap \centre(B,{\cdot})=\Soc(B)\cap \Fix(B).
\end{align*}
It is rather easy to prove that $\Fix(B)$ is a left ideal of $B$, while $\Soc(B)$ and $\centre(B)$ are ideals of $B$ (see \cite[Lemma~2.5]{GuarnieriVendramin17} and \cite[remark after Definition~7]{CatinoColazzoStefanelli19}, respectively). The kernel of lambda is a subbrace of~$B$, but in general it is not a left ideal of~$B$.

Given $X\subseteq B$, we define the \emph{ideal generated} by $X$, denoted by $\left\langle X\right\rangle^B$, as the smallest ideal of $B$ that contains $X$.
The \emph{commutator} of two ideals $I$ and $J$ is defined in \cite{BournFacchiniPompili23} as
\[[I,J]^B=\left\langle [I,J]_+\cup [I,J]_\bullet\cup\{ij-(i+j)\;|\;i\in I,j\in J\}\right\rangle^B.\]
The \emph{derived ideal} $[B,B]^B$ of $B$ is an important ideal of $B$ that plays a key role in the study of nilpotency and solubility of skew braces (see \cite{BallesterEstebanFerraraPerezCTrombetti25-pacificjm-cent-nilp, BallesterEstebanJimenezPerezC24-solubleskewbraces}). The results of these papers show that the derived ideal is the true analogue of the group commutator. In fact, $[B,B]^B$ is the smallest ideal of $B$ with abelian quotient, and $[B,B]^B=B\ast B+[B,B]_+=(B\ast B)[B,B]_\bullet$ (see \cite[Proposition~3.3]{BallesterEstebanPerezA26-jalg-prod-ab}). We say that a skew brace $B$ is \emph{perfect} if $B=[B,B]^B$.

In \cite{BonattoJedlicka23} the authors define skew brace analogues of the upper and lower central series of groups. They define the \emph{upper central series} of $B$ as
\[0=\centre_0(B)\subseteq \centre_1(B)\subseteq\cdots\subseteq \centre_n(B)\subseteq\cdots,\]
where
\[\centre_{n+1}(B)/{\centre_{n}(B)}=\centre(B/{\centre_{n}(B)}),\quad n\in\mathbb{N}\cup\{0\},\]
and they define the \emph{lower central series} of $B$ as
\[B=\Gamma_1(B)\supseteq \Gamma_2(B)\supseteq\cdots\supseteq\Gamma_n(B)\supseteq\cdots,\]
where
\[\Gamma_{n+1}(B)=[\Gamma_n(B),B]^B=\langle\Gamma_n(B)\ast B+B\ast \Gamma_n(B)+[\Gamma_n(B),B]_+\rangle^B,\quad n\in\mathbb{N}.\]
The second equality follows from \cite[Theorem~3.3]{BallesterEstebanFerraraPerezCTrombetti25-pacificjm-cent-nilp}.

The skew brace $B$ is said to be \emph{centrally nilpotent} (see~\cite{BonattoJedlicka23}) if there exists an $n \in \mathbb{N}$ such that $\centre_n(B) = B$ (equivalently, $\Gamma_{n+1}(B)=0$, by \cite[Proposition~4.2]{BallesterEstebanFerraraPerezCTrombetti25-pacificjm-cent-nilp}). The smallest such $n$ is called the \emph{central nilpotency class} of~$B$.

A two-sided skew brace is a skew-brace in which a right version $(b+c)a=ba-a+ca$ for all $a$, $b$, $c\in B$ of the left distributive-like property holds. This stronger condition imposes restrictions on the underlying algebraic structures not present  in general skew braces and provides a natural setting to study the interaction between the behaviours of the multiplicative and the additive groups. For instance, two-sided skew braces satisfy the nice property that the multiplicative conjugation by $a\in B$, which is an automorphism of $(B, {\cdot})$, is also an automorphism of $(B, {+})$, that is, for every $a$, $x$, $y\in B$, we have that $a(x+y)a^{-1}=axa^{-1}+aya^{-1}$ (see \cite[Lemma~4.1]{Nasybullov19}). Moreover, two-sided skew braces with abelian additive group correspond exactly to radical rings (see~\cite{Rump07}).

The main result of \cite{Tsang26-jmathsocjapan} is the following one (see Theorem~1.6 and Corollaries~1.8 and~2.6 of that paper).
\begin{teo}\label{teo-grun-tsang}
Let $B$ be a skew brace such that $B\ast B=B$. Then $\centre_2(B)=\centre(B)$ if and only if, the multiplicative conjugation by every element of $\centre_2(B)$ is an automorphism of $(B,+)$. 

In particular, if $B$ is a two-sided skew brace such that $B\ast B=B$, then $\centre(B/{\centre(B)})=1$.
\end{teo}

Observe that Grün's lemma may be recovered from Theorem~\ref{teo-grun-tsang} by considering for a group $(G, {\cdot})$ the two-sided skew brace $(G,{{\cdot}^{\text{op}}},{\cdot})$, with ${\cdot}^{\text{op}}$ the opposite operation of $\cdot$ defined by means of $a\cdot^{\text{op}}b=b\cdot a$. For this skew brace, $G*G$ coincides with the derived subgroup of~$G$ and $\ker \lambda=\centre(G, {\cdot})=\centre(G, {\cdot^{\text{op}}})$, consequently $\centre(G, {\cdot^{\text{op}}}, {\cdot})=\centre(G, {\cdot})=\centre(G, {\cdot^{\text{op}}})$ and, analogously, $\centre_2(G, {\cdot}^{\text{op}}, {\cdot})=\centre_2(G, {\cdot})$.




The first main result of this paper is a stronger version of Theorem~\ref{teo-grun-tsang}, in which we replace $B*B$ by the derived ideal.

\begin{athm}\label{teo-grun}
Let $B$ be a perfect skew brace. Then $\centre_2(B)=\centre(B)$ if, and only if, the multiplicative conjugation by every element of $\centre_2(B)$ is an automorphism of $(B,+)$. 

In particular, if $B$ is a two-sided perfect skew brace, then $\centre(B/{\centre(B)})=1$.
\end{athm}

The following skew brace analogue of Schur's theorem has been proved in \cite[Theorem~5.4]{JespersKubatVanAntwerpenVendramin21}.
\begin{teo}\label{teo-schur-JKVV}
Let $B$ be a skew brace such that $B/{\centre(B)}$ is finite. Then $B*B$ and $[B, B]_+$ are finite. In particular, $[B,B]^B$ is finite.
\end{teo}

To prove a skew brace version of Baer's theorem for the case $n>1$, we have to add another hypothesis.
\begin{athm}\label{teo-baer}
Let  $B$ be a skew brace such that $[B, B]^B$ is contained in $\ker\lambda$ and  $B/{\centre_n(B)}$ is finite. Then $\Gamma_{n+1}(B)$ is finite.
\end{athm}

Baer's theorem may be recovered from Theorem~\ref{teo-baer} when $B$ is trivial, because $\ker\lambda=B$ in this case and so the lower and the upper central series of $B$ coincide with those of the underlying group $(B, {+})$. We leave as an open question whether the condition ``$[B, B]^B$ is contained in $\ker\lambda$'' can be dropped in Theorem~\ref{teo-baer}.

A natural follow-up question of this result is whether we can bound the order of $\Gamma_{n+1}(B)$ by a function of the order of $B/{\centre_n(B)}$. Let us define, for every $t\in\mathbb{N}$, $\omega(t)=t^m$, with $m=(\log_p(t)-1)/2$ and $p$ the smallest prime divisor of $t$. We note that if $s$, $t\in \N$ and $s\mid t$, then $\omega(s)\le \omega(t)$. In the case of groups and $n=1$, we have some bounds (see \cite[Section 3]{Wiegold65}).
\begin{teo}\label{teo-schur-bound-weig}
Let $G$ be a group such that $|G/{\centre(G)}|=t$. Then $|G'|\leq\omega(t)$.
\end{teo}

We can generalise this result to skew braces.
\begin{athm}\label{teo-schur-bound}
Let $B$ be a skew brace such that $|B/{\centre(B)}| = t$. Then $[B,B]^B$ is finite and
\[|[B,B]^B|\leq \frac{\omega(t^2/\lvert\ker\lambda:(\centre(B),\cdot)\rvert)}{|((B,\cdot)/{\ker\lambda})'|}\le \omega(t^2).\]
\end{athm}

Observe that when we apply Theorem~\ref{teo-schur-bound} to a trivial skew brace, we recover Theorem~\ref{teo-schur-bound-weig}. 

In case that $(B,{+})$ is abelian, we obtain a bound for $\lvert \Gamma_{n+1}(B)\rvert$ as a function of $\lvert B/{\centre_n(B)}\rvert$.
\begin{athm}\label{teo-baer-bound}
Let $B$ be a skew brace such that $(B, {+})$ is abelian. 
If $[B,B]^B$ is contained in $\ker\lambda$ and $|B/{\centre_n(B)}|=t$, then $\Gamma_{n+1}(B)$ is finite and $|\Gamma_{n+1}(B)|\leq {\omega(t^2)}^{t^{n-1}}$.
\end{athm}

We leave as an open question whether the facts that $(B, {+})$ is abelian or $[B, B]^B$ is contained in $\ker \lambda$ can be dropped in Theorem~\ref{teo-baer-bound}.


\section{Preliminaries on trifactorised groups}
We shall adhere to the notation of \cite{BallesterEstebanPerezAPerezC26-commmathstat-categoriesskewleftbraces}. For the reader's convenience, we outline in this section the notation and the results that we will need for the proofs of our theorems.

Given a skew brace $(B,+,\cdot)$, we denote by $K=(B,+)$ its additive group and by $C=(B,\cdot)$ its multiplicative group. The identity map $\delta\colon C\longrightarrow K$ is a bijective derivation with respect to the lambda action, this means that $\delta(ce)=\delta(c)+\lambda_c(\delta(e))$ for all $c$, $e\in C$. 

Consider $\eta\colon C\longrightarrow E$ a group epimorphism such that $\ker\eta\leq \ker\lambda$. We can define a group homomorphism $\bar\lambda\colon E\longrightarrow\Aut(K)$ such that $\bar\lambda(\eta(c))=\lambda(c)$ for all $c\in C$. Hence, $E$ acts on $K$ via automorphisms by means of $\bar\lambda$, and so we can consider the corresponding semidirect product $G=[K]E$ (see~\cite{BallesterEstebanPerezAPerezC26-commmathstat-categoriesskewleftbraces}).

We identify $K$ with the normal subgroup $\{(k, 1)\;|\; k\in K\}$ and $E$ with the subgroup $\{(0, \eta(c))\;|\; c\in C\}$ of~$G$. For the sake of simplicity, we will denote an element $(k, e)$ of the group $G$, with $k\in K$ and $e\in E$, in the form $ke$. Then $G$ can be viewed as a split extension $G = KE$ and $K \cap E = 1$, and the lambda action of $E$ on $K$ is just the conjugation action within $G$. Furthermore, $G$ possesses a triple factorisation $G=KE=KH=HE$, where $H=\{\delta(c)\eta(c)\;|\; c\in C\}$ is a subgroup of $G$, and $K\cap E=H\cap E=1$. We say that $(G,K,H,E)$ is a \emph{(generalised) trifactorised group} associated with $B$ (see \cite{BallesterEstebanPerezAPerezC26-commmathstat-categoriesskewleftbraces}).


There are  two particular trifactorised groups associated with $B$ that will be relevant for this work. If we consider $\eta=\id_C\colon C\longrightarrow C$, we have the \emph{large trifactorised group} associated with~$B$, and if we consider $\eta\colon C\longrightarrow E=\{(0,\lambda_c)\;|\;c\in C\}\leq [K](\Aut(K))$ (the semidirect product of $K$ by its automorphism group) given by $\eta(c)=(0,\lambda_c)$,  we have the \emph{small trifactorised group} associated with $B$. We observe that $E$ coincides with the image of $\lambda$ in $\Aut(K)$.

The following results present the relations between the subbrace and ideal structure of a skew brace and the subgroup structure of a fixed, but arbitrarily chosen, associated trifactorised group $(G,K,H,E)$ associated with a skew brace $B$.

\begin{prop}[{\cite[Propositions~5.3~(1), 5.3~(4), and~5.5]{BallesterEstebanPerezAPerezC26-commmathstat-categoriesskewleftbraces}}]
Let $L$ be a subset of $K$. Then $L$ is a subbrace of $B$ if and only if, $T=LH\cap LE\leq G$. In this case, 
\[(T,T\cap K,T\cap H,T\cap E)\]
is a trifactorised group associated with $L$.

Moreover, $L$ is an ideal of $B$ if and only if, $T$ is normal in $G$.
\end{prop}
Therefore, $LH\cap LE$ is called the \emph{trifactorised subgroup} of $G$ associated with the subbrace $L$. Note that, by \cite[Proposition~6.6]{BallesterEstebanPerezAPerezC26-commmathstat-categoriesskewleftbraces}, $LH\cap LE=L\delta^{-1}(L)$.

\begin{prop}[{\cite[Proposition~6.6]{BallesterEstebanPerezAPerezC26-commmathstat-categoriesskewleftbraces}}]\label{prop-3fact-quo}
Consider an ideal $I$ of $B$ and the trifactorised subgroup $T$ of $G$ associated with $I$. Then 
\[(G/T,KT/T,HT/T,ET/T)\]
is a trifactorised group associated with the skew brace $B/I$ with 
epimorphism $\bar\eta\colon C/D\longrightarrow ET/T$ given by $\bar\eta(cD)=\eta(c)T$, where $D=(I,\cdot)$.
\end{prop}

\begin{lemma}\label{lemma-sub-3fact-quo-ass}
Let $I$ be an ideal of $B$ and $A$ a subbrace of $B$. Let $T_0$ and $T$ be the trifactorised subgroups of $(G,K,H,E)$ associated with $I$ and $A$, respectively. Then the trifactorised subgroup of $(G/T_0,KT_0/T_0,HT_0/T_0,ET_0/T_0)$ associated with the subbrace $AI/I$ of $B/I$ is $TT_0/T_0$. 
\end{lemma}
\begin{proof}
Let $L$ be the subgroup of $G$ corresponding to the additive subgroup of $A$ and denote $(\eta\circ \delta^{-1})(L)=M$. The additive and multiplicative subgroups of $AI/I$ correspond, respectively, to the subgroups
\begin{align*}
&\{\bar\delta(c(I,\cdot))\;|\;c\in (A,\cdot)\}=\{\delta(c)T_0\;|\;c\in (A,\cdot)\}=LT_0/T_0,\\
&\{\bar\eta(c(I,\cdot))\;|\;c\in (A,\cdot)\}=\{\eta(c)T_0\;|\;c\in (A,\cdot)\}=MT_0/T_0
\end{align*}
 of $G/T_0$. Hence, the trifactorised subgroup of $G/T_0$ associated with $AI/I$ is $(LT_0/T_0)(MT_0/T_0)=LMT_0/T_0=TT_0/T_0$.
\end{proof}

 We compute the conjugation $\conj{b}a = bab^{-1}$ and the commutator $[a, b] = aba^{-1}b^{-1}$, for $a, b \in G$, in the large trifactorised group $G$ associated with the skew brace $B$. 
 
Let $A_1$ and $A_2$ be subbraces of $B$. We identify them with the subgroups $L_1$ and $L_2$ of $K$ in $G$, respectively. If $a \in A_1$ and $b \in A_2$, then $a\ast b$ corresponds to the element $[\eta(\delta^{-1}(a)), b]$ of $G$, which belongs to $K$. Therefore, the additive group of $A_1\ast A_2$ can be identified  as the subgroup $[(\eta\circ\delta^{-1})(L_1),L_2]$ of $G$. Note that $[(\eta\circ\delta^{-1})(L_1),L_2]$ is contained in $K$. By \cite[Lemma~5.4]{BallesterEstebanPerezAPerezC26-commmathstat-categoriesskewleftbraces}, $(\eta\circ\delta^{-1})(L_1)=L_1^{-1}H\cap E=L_1H\cap E$. Therefore, we have

\begin{prop}\label{prop-add-gr-deriv}
The subgroup $A_1\ast A_2$ of $K$ corresponds to the subgroup $[(\eta\circ\delta^{-1})(L_1), L_2]=[L_1H\cap E,L_2]$ of $G$. In particular, the additive group of $B\ast B$ corresponds to the normal subgroup $[E,K]$ of $G$.
\end{prop}

Let us identify other important substructures with subgroups of $K$ or~$E$.
\begin{prop}\label{prop-sub-in-3fact}
\begin{enumerate}
\item The additive group of $[B,B]^B$ corresponds to the subgroup $K'[E,K]$ of $G$ and the multiplicative subgroup of $[B,B]^B$ corresponds to the subgroup $E'([E,K]H\cap E)$ of $G$.\label{en-prop-sub-in-3-fact-1}
\item The additive group of $\Fix(B)$ corresponds to the subgroup $\Cent_K(E)$ of~$G$.\label{en-prop-sub-in-3-fact-2}
\item The multiplicative group of $\ker\lambda$ corresponds to the subgroup $\Cent_E(K)$ of~$G$.\label{en-prop-sub-in-3-fact-3}
\end{enumerate}
\end{prop}
\begin{proof}
We have that $[B,B]^B=B\ast B+[B,B]_+=(B\ast B)[B,B]_\bullet$. By Proposition~\ref{prop-add-gr-deriv}, the additive group of $B\ast B$ corresponds to the subgroup $[E,K]$ of $G$ and, by \cite[Lemma~5.4]{BallesterEstebanPerezAPerezC26-commmathstat-categoriesskewleftbraces}, the multiplicative group of $B\ast B$ correspond to the subgroup $(\eta\circ\delta^{-1})([E,K])=[E,K]H\cap E$ of $G$. It is clear that the additive group of $[B,B]_+$ corresponds to the subgroup $K'$ and the multiplicative group of $[B,B]_\bullet$ corresponds to $E'$. Therefore, Statement~\ref{en-prop-sub-in-3-fact-1} follows.

Since $\bar\lambda$ corresponds to the conjugation of the elements of $K$ by the elements of $E$, the elements of $\Fix(B)$ are the elements of $K$ that commute with all elements of $E$. In addition, the elements of $\ker\lambda$ are the elements of $E$ that commute with all elements of $K$. Hence, Statements~\ref{en-prop-sub-in-3-fact-2} and~\ref{en-prop-sub-in-3-fact-3} follow.\qedhere
\end{proof}

\section{Central series and trifactorised groups}
Let $B$ be a skew brace and consider a trifactorised subgroup $(G,K,H,E)$ associated with $B$. The objective of this section is to study the relationship between the trifactorised subgroups of $G$  associated with $\centre_n(B)$, $\Gamma_n(B)$ and $\centre_n(G)$, $\gamma_{n}(G)$, respectively. Let us denote by $T_n^Z$ and $T_n^\Gamma$ the trifactorised subgroups associated with $\centre_n(B)$ and $\Gamma_n(B)$, respectively.

\begin{lemma}\label{lemma-3fact-gamma-fact}
Denote by $\gamma_{n+1}^K(G)=\gamma_{n+1}(G)\cap K$. Then
\begin{align*}
\gamma_{n+1}(G)&=\gamma_{n+1}^K(G)\gamma_{n+1}(E),\\
\gamma_{n+1}^K(G)&=[\gamma_n^K(G), K][\gamma_n^K(G), E][\gamma_n(E), K]
\end{align*}
for $n\geq 1$.
\end{lemma}
\begin{proof}
For $n=1$ the result is obvious since $G'=K'[E,K]E'$, so let us suppose it is true for $n$. By definition,
\[\gamma_{n+2}(G)=[\gamma_{n+1}(G),G]=[\gamma_{n+1}^K(G)\gamma_{n+1}(E), KE].\]
 Since $\gamma_{n+1}(E)\leq \Norm_G(KE)\cap \Norm_G(\gamma_{n+1}^K(G))$ and $E\leq \Norm_G(K)\cap \Norm_G(\gamma_{n+1}(E))\cap \Norm_G(\gamma_{n+1}^K(G))$, by \cite[Kapitel III, Hilfssatz~1.10]{Huppert67} it follows that
 \begin{align*}
   \gamma_{n+2}(G)&=
                    [\gamma_{n+1}^K(G)\gamma_{n+1}(E), KE]\\&=[\gamma_{n+1}^K(G), KE][\gamma_{n+1}(E), KE]\\&=[\gamma_{n+1}^K(G), K][\gamma_{n+1}^K(G), E][\gamma_{n+1}(E), K][\gamma_{n+1}(E), E]\\&=[\gamma_{n+1}^K(G), K][\gamma_{n+1}^K(G), E][\gamma_{n+1}(E), K]\gamma_{n+2}(E).
 \end{align*}
 By Dedekind's identity,
 \begin{align*}
   \gamma_{n+2}^K(G)&=\gamma_{n+2}(G)\cap K\\
                    &=[\gamma_{n+1}^K(G), K][\gamma_{n+1}^K(G), E][\gamma_{n+1}(E), K]\gamma_{n+2}(E)\cap K\\
                    &=[\gamma_{n+1}^K(G), K][\gamma_{n+1}^K(G), E][\gamma_{n+1}(E), K](K\cap \gamma_{n+2}(E))\\
   &=[\gamma_{n+1}^K(G), K][\gamma_{n+1}^K(G), E][\gamma_{n+1}(E), K]
 \end{align*}
 and the result follows.\qedhere
\end{proof}

\begin{lemma}\label{lemma-deriv-3fact}
$T_2^\Gamma$ contains $G'$ and $T_2^\Gamma\cap K=G'\cap K$.
\end{lemma}
\begin{proof}
By Proposition~\ref{prop-sub-in-3fact}, the trifactorised subgroup associated with $\Gamma_2(B)=[B,B]^B$ is $T_2^\Gamma=K'[E,K]E'([E,K]H\cap E)$, therefore, $T_2^\Gamma$ contains $K'[E,K]E'$, which is equal to $G'$. Moreover, $T_2^\Gamma\cap K=K'[E,K]=G'\cap K$, by Dedekind's identity.\qedhere
\end{proof}

\begin{teo}\label{teo-central-ser-3fact}
\begin{enumerate}
\item $T_n^Z$ is contained in $\centre_n(G)$ for $n\geq 1$.
\item $T_{n+1}^\Gamma$ contains $\gamma_{n+1}(G)$ for $n\geq 1$.
\end{enumerate}
\end{teo}
\begin{proof}
\begin{enumerate}
\item Since $\centre(B)\leq \centre(B,+)\cap \Fix(B)$, by Proposition~\ref{prop-sub-in-3fact}, it follows that $T_1^Z\cap K\leq \centre(K)\cap \Cent_K(E)=\Cent_K(G)\leq \centre(G)$. On the other hand, since $\centre(B)\leq \ker\lambda\cap \centre(B,\cdot)$, by Proposition~\ref{prop-sub-in-3fact}, it follows that $T_1^Z\cap E\leq \centre(E)\cap \Cent_E(K)=\Cent_E(G)\leq \centre(G)$. Therefore,
  \[T_1^Z=(T_1^Z\cap K)(T_1^Z\cap E)\leq \centre(G).\]
Thus, the result is true for $n=1$.

By Lemma~\ref{lemma-sub-3fact-quo-ass} the trifactorised subgroup of \[(G/T_n^Z, KT_n^Z/T_n^Z,HT_n^Z/T_n^Z,ET_n^Z/T_n^Z)\] associated with the ideal $\centre_{n+1}(B)/{\centre_n(B)}=\centre(B/{\centre_n(B)})$ is $T_{n+1}^Z/T_n^Z$. Applying the case for $n=1$, we have that $T_{n+1}^Z/T_n^Z\leq \centre(G/T_n^Z)$, hence $\{T_n^Z\}_{n\in\N}$ is an ascending central series of $G$, hence, by \cite[Kapitel III, Satz~2.7]{Huppert67}, $T_n^Z\leq \centre_n(G)$.

\item  By Lemma~\ref{lemma-deriv-3fact}, the result is true for $n=1$. Suppose that it is true for $n-1$. By definition and \cite[Theorem~3.3]{BallesterEstebanFerraraPerezCTrombetti25-pacificjm-cent-nilp}, $\Gamma_{n+1}(B)$ contains $[\Gamma_{n}(B),B]_+$, $[\Gamma_{n}(B),B]_\bullet$, $\Gamma_{n}(B)\ast B$ and $B\ast \Gamma_{n}(B)$.

The additive subgroups of $[\Gamma_{n}(B),B]_+$, $\Gamma_{n}(B)\ast B$, and $B\ast \Gamma_{n}(B)$ correspond to the subgroups $[T_{n}^\Gamma\cap K,K]$, $[T_{n}^\Gamma\cap E, K]$, $[T_{n}^\Gamma\cap K, E]$ of $K\le G$, respectively, where the last two follow from Proposition~\ref{prop-add-gr-deriv}. The multiplicative subgroup of  $[\Gamma_{n}(B),B]_\bullet$ corresponds to the subgroup $[T_n^\Gamma\cap E, E]$ of $E\le G$. Therefore, $T_{n+1}^\Gamma\cap K$ contains $[T_{n}^\Gamma\cap K,K]$, $[T_{n}^\Gamma\cap K, E]$, $[T_{n}^\Gamma\cap E, K]$ and $T_{n+1}^\Gamma\cap E$ contains $[T_{n}^\Gamma\cap E, E]$.
 
By Lemma~\ref{lemma-3fact-gamma-fact}, $T_n^\Gamma$ contains $\gamma_n(G)=\gamma_n^K(G)\gamma_n(E)$, hence, $T_{n}^\Gamma\cap K$ and $T_{n}^\Gamma\cap E$ contain $\gamma_{n}^K(G)$ and $\gamma_{n}(E)$, respectively. Therefore, by Lemma~\ref{lemma-3fact-gamma-fact}, $T_{n+1}^\Gamma$ contains
\begin{align*}
&[T_{n}^\Gamma\cap K,K][T_{n}^\Gamma\cap K, E][T_{n}^\Gamma\cap E, K][T_n^\Gamma\cap E,E]\\&\geq [\gamma_n^K(G),K][\gamma_n^K(G), E][\gamma_n(E), K][\gamma_n(E),E]=\gamma_{n+1}^K(G)\gamma_{n+1}(E)\\&=\gamma_{n+1}(G).\qedhere
\end{align*}
\end{enumerate}
\end{proof}

%

\begin{prop}\label{prop-des-cent-cap-K-eq}
If $n=1$ or $[B,B]^B$ is contained in $\ker\lambda$, then $\gamma_{n+1}(G)\cap K=T_{n+1}^\Gamma\cap K$.
\end{prop}
\begin{proof}
By Lemma~\ref{lemma-deriv-3fact}, the result is true for $n=1$. Suppose it is true for $n$. Since $\Gamma_2(B)\subseteq\ker\lambda$ it follows that $\Gamma_{n+1}(B)\ast B=0$ for $n\geq 1$. Then 
\begin{align*}
\Gamma_{n+2}(B)&=\Gamma_{n+1}(B)\ast B+B\ast \Gamma_{n+1}(B)+[\Gamma_{n+1}(B), B]_+\\&=B\ast \Gamma_{n+1}(B)+[\Gamma_{n+1}(B), B]_+.
\end{align*}
Therefore, by Proposition~\ref{prop-add-gr-deriv}, the additive subgroup of $\Gamma_{n+2}(B)$ corresponds to the subgroup 
\[T_{n+2}^\Gamma\cap K=[E,T_{n+1}^\Gamma\cap K][T_{n+1}^\Gamma\cap K,K].\]
By induction, $T_{n+1}^\Gamma\cap K\leq \gamma_{n+1}(G)$, therefore,
\[T_{n+2}^\Gamma\cap K\leq [E,\gamma_{n+1}(G)][\gamma_{n+1}(G),K]\leq \gamma_{n+2}(G).\] 
Hence, $T_{n+2}^\Gamma\cap K\leq \gamma_{n+2}(G)\cap K$. On the other hand, by Theorem~\ref{teo-central-ser-3fact}, $\gamma_{n+2}(G)\cap K\leq T_{n+2}^\Gamma\cap K$. The result follows.\qedhere
\end{proof}

\section{Proof of Theorem~\ref{teo-grun}}

We prove a stronger version of \cite[Theorem 1.6]{Tsang26-jmathsocjapan}.
\begin{prop}\label{prop-Z2-commut}
Let $B$ be a skew brace. Then
\begin{gather*}
\centre_2(B)\ast[B,B]^B=0,\\
[\centre_2(B),[B,B]^B]_+=0,\\
[B,B]_\bullet\ast \centre_2(B)=0.
\end{gather*}
\end{prop}
\begin{proof}
Consider the large trifactorised group $(G,K,D,C)$ associated with $B$ and let $T_2^Z$ be the trifactorised subgroup associated with $\centre_2(B)$. By Proposition~\ref{prop-sub-in-3fact}, the additive group of $[B,B]^B$ corresponds to $K'[C,K]$.

By Proposition~\ref{prop-add-gr-deriv}, the subgroups associated with the additive groups of $\centre_2(B)\ast[B,B]^B$, $[\centre_2(B),[B,B]^B]_+$ and $[B,B]_\bullet\ast \centre_2(B)$ are
\[[T_2^Z\cap C,K'[C,K]],\quad [T_2^Z\cap K,K'[C,K]],\quad[C',T_2^Z\cap K],\]
 respectively. 
 
By Theorem~\ref{teo-central-ser-3fact}, $T_2^Z\subseteq \centre_2(G)$. Since $C'$ and $K'[C,K]$ are contained in $G'$, it follows that our three commutators are contained in $[\centre_2(G),G']=0$. (see \cite[Kapitel~III, Hauptsatz~2.11]{Huppert67}).\qedhere
\end{proof}

The following well-known identities, that can be easily checked, will be applied without further reference.
\begin{lemma}
  Let $B$ be a skew brace and let $a$, $b$, $c\in B$. Then:
  \begin{enumerate}
  \item $a*(b+c)=a*b+b+a*c-b$.
    
  \item $a*(-b)=-b-a*b+b$. If $b$ commutes with $a*b$ in $(B, {+})$, then $a*(-b)=-a*b$.
  \item $(ab)*c=a*(b*c)+b*c+a*c$.
  \end{enumerate}
\end{lemma}
\begin{prop}\label{prop-equi-grun}
Let $B$ be a perfect skew brace. Then the following statements are equivalent:
\begin{enumerate}
\item $\centre_2(B)=\centre(B)$.\label{prop-equi-grun-1}
\item $(B\ast B)\ast \centre_2(B)=0$.\label{prop-equi-grun-2}
\end{enumerate}
\end{prop}
\begin{proof}
\ref{prop-equi-grun-1} implies \ref{prop-equi-grun-2}. It follows from the fact that $\centre_2(B)=\centre(B)\subseteq \Fix(B)$.

\ref{prop-equi-grun-2} implies \ref{prop-equi-grun-1}. We want to prove that $\centre_2(B)\subseteq \centre(B)$. Since $B$ is perfect, Proposition~\ref{prop-Z2-commut} yields that $\centre_2(B)\ast B=[\centre_2(B),B]_+=0$. Hence, $\centre_2(B)\subseteq \centre(B,{+})\cap\ker\lambda=\Soc(B)$. Thus, we only need to prove that $\centre_2(B)$ is contained in $\Fix(B)$. By Proposition~\ref{prop-Z2-commut}, $[B,B]_\bullet\ast \centre_2(B)=0$. Recall that $B=[B,B]^B=(B\ast B)[B,B]_\bullet$. Let $b\in B$ and $z\in \centre_2(B)$. There exist $a\in B\ast B$ and $c\in [B,B]_\bullet$ such that $b=ac$. Then 
\[b\ast z=(ac)\ast z=a\ast (c\ast z)+c\ast z+a\ast z=0.\]
 Therefore, $B\ast \centre_2(B)=0$, namely, $\centre_2(B)\subseteq\Fix(B)$. \qedhere
\end{proof}

\begin{prop}[{\cite[Section~2.3]{Tsang26-jmathsocjapan}}]\label{prop-tsang-2.3}
Let $B$ be a skew brace. The following statements are equivalent:
\begin{enumerate}
\item $(B\ast B)\ast \centre_2(B)=0$.
\item The multiplicative conjugation by every element of $\centre_2(B)$ is an automorphism of $(B,{+})$.
\end{enumerate}
\end{prop}

\begin{proof}[{Proof of Theorem~\ref{teo-grun}}]
It follows directly from Propositions~\ref{prop-equi-grun} and \ref{prop-tsang-2.3}. Furthermore, in a two-sided skew brace the inner automorphisms of $(B,{\cdot})$ are automorphisms of $(B,{+})$ (see \cite[Lemma~4.1]{Nasybullov19}).\qedhere
\end{proof}

\section{Proof of Theorem~\ref{teo-baer}}


\begin{proof}[{Proof of Theorem~\ref{teo-baer}}]
Consider $(G,K,D,C)$ the large trifactorised group associated with $B$ and let $T_n^Z$ and $T_n^\Gamma$ be the trifactorised subgroups associated with $\centre_n(B)$ and $\Gamma_n(B)$, respectively.

Since $B/{\centre_n(B)}$ is finite, $K/(T_n^Z\cap K)$ and $C/(T_n^Z\cap C)$ are finite, therefore, by \cite[Lemma~1.2.5]{AmbergFranciosiDeGiovanni92}, $G/T_n^Z$ is finite. By Theorem~\ref{teo-central-ser-3fact}, $T_n^Z\leq \centre_n(G)$, hence, $G/{\centre_n(G)}$ is finite. Applying Baer's theorem, we conclude that $\gamma_{n+1}(G)$ is finite.

Applying Proposition~\ref{prop-des-cent-cap-K-eq}, we have that $T_{n+1}^\Gamma\cap K=\gamma_{n+1}(G)\cap K\leq\gamma_{n+1}(G)$. Therefore, $T_{n+1}^\Gamma\cap K$ is finite, which means that $\Gamma_{n+1}(B)$ is finite.\qedhere
\end{proof}

\section{Proof of Theorem~\ref{teo-schur-bound}}
\begin{proof}[{Proof of Theorem~\ref{teo-schur-bound}}]
Consider $(G,K,H,E)$ the small trifactorised group associated with $B$. Let $T_1^Z$ be the trifactorised subgroup associated with $\centre(G)$. Consider $C=(B,\cdot)$ and $C_1^Z=(\centre(B),\cdot)$.

We have that $K/(T_1^Z\cap K)\cong (B,+)/(\centre(B),+)$, therefore, $|K/(T_1^Z\cap K)|=t$. On the other hand, $E\cong C/{\ker\lambda}$ and $T_1^Z\cap E\cong C_1^Z\ker\lambda/{\ker\lambda}$ via the epimorphism $\eta=\lambda\colon C\longrightarrow E$. Therefore,
\[E/(T_1^Z\cap E)\cong C/C_1^Z\ker\lambda\cong (C/C_1^Z)/(C_1^Z\ker\lambda/C_1^Z).\]
Since $C/C_1^Z\cong (B,\cdot)/(\centre(B),\cdot)$, it is finite. Moreover, 
\begin{align*}
  C_1^Z\ker\lambda/C_1^Z&\cong \ker\lambda/(\ker\lambda\cap C_1^Z)\\&\cong \ker\lambda/(\ker\lambda\cap (\centre(B),\cdot))=\ker\lambda/(\centre(B),\cdot),
\end{align*}
as $\centre(B)$ is contained in $\ker\lambda$.
Therefore, $|E/(T_1^Z\cap E)|=t/l$, where $l=|{\ker\lambda}:{\centre(B)}|$. Applying \cite[Lemma~1.2.5]{AmbergFranciosiDeGiovanni92}, we have that $|G:T_1^Z|\leq t^2/l$. By Theorem~\ref{teo-central-ser-3fact}, $T_1^Z\leq \centre(G)$, therefore, $|G:\centre(G)|\leq t^2/l$. Hence, by Theorem~\ref{teo-schur-bound-weig}, $|G'|\leq\omega(t^2/l)$.

By Proposition~\ref{prop-sub-in-3fact}, the additive subgroup of $[B,B]^B$ corresponds to the subgroup $K'[E,K]$ of $G$. Hence, $|[B,B]^B|=|K'[E,K]|=|G'|/|E'|\leq \omega(t^2/l)/|E'|$. Finally, recall that $E\cong (B,\cdot)/{\ker\lambda}$.\qedhere
\end{proof}

\section{Proof of Theorem~\ref{teo-baer-bound}}

\begin{lemma}\label{lemma-add-power-star}
Let $B$ be a skew brace such that $(B, {+})$ is abelian. Then $z(a\ast b)=a\ast (zb)$ for all $a,b\in B$ and $z\in\Z$.
\end{lemma}
\begin{proof}
  If $z\ge 0$, the result follows by induction on $z$, since $a\ast (0b)=a\ast 0=0=0(a\ast b)$ and, for $z\ge 0$, if $a\ast (zb)=z(a\ast b)$, then  $a\ast ((z+1)b)=a\ast(zb+b)=a\ast (zb)+b+a\ast b-b=z(a\ast b)+a\ast b=(z+1)(a\ast b)$. If $z<0$, then $a\ast (zb)=-(a\ast (-(zb)))=-(a\ast ((-z)b))=-(-z)(a\ast b)=z(a\ast b)$.\qedhere
\end{proof}


We present a proof of the following proposition using trifactorised groups, but it also follows from \cite[Theorem~6.3]{Tsang25-commmath}.
\begin{prop}\label{prop-cent-left-ser-star-0}
Let $B$ be a skew brace. Then
\[\centre_n(B)\ast B^n=0,\;n\geq 1.\]
\end{prop}
\begin{proof}
Consider $(G,K,D,C)$ the large trifactorised group associated with $B$ and let $T_n^Z$ be the trifactorised subgroup associated with $\centre_n(B)$. By Proposition~\ref{prop-add-gr-deriv}, the additive group associated with $B^n$ is $L_n^\ast=[C,[\cdots [C,K]\cdots]]$, where $C$ appears $n-1$ times. By Proposition~\ref{prop-add-gr-deriv}, the additive subgroup associated with $\centre_n(B)\ast B^n$ is $[T_n^Z\cap E,L_n^\ast]$. Note that $L_n^\ast$ is contained in $\gamma_n(G)$ and, by Theorem~\ref{teo-central-ser-3fact}, $T_n^Z$ is contained in $\centre_n(G)$, hence, $[T_n^Z\cap E,L_n^\ast]\subseteq[\centre_n(G),\gamma_n(G)]=0$ (see \cite[Kapitel~III, Hauptsatz~2.11]{Huppert67}).\qedhere
\end{proof}

\begin{teo}\label{teo-left-ser-baer-ab-typ}
  Let $B$ be a skew brace such that $(B, {+})$ is abelian and $|B/{\centre_n(B)}|=t$. Then $B^{n+1}$ is finite and $|B^{n+1}|\leq {\omega(t^2)}^{t^{n-1}}$.
\end{teo}

\begin{proof}
  By Theorem~\ref{teo-schur-bound}, the result is true for $n=1$, so let us suppose that it is true for $n-1$. Then
\begin{align*}B^{n+1}&=B\ast B^n\\&=\Big\{\sum_{1\leq i\leq m}z_i(b_i\ast x_i)\;\Big|\; z_i\in\Z,\; b_i\in B,\; x_i\in B^n,\; 1\le i\le m,\; m\in \N\Big\}.
\end{align*}
By Lemma~\ref{lemma-add-power-star}, $z_i(b_i\ast x_i)=b_i\ast (z_ix_i)$. Therefore, 
\[B^{n+1}=\Big\{\sum_{1\leq i\leq m}b_i\ast x_i\;\Big|\; b_i\in B,\; x_i\in B^n,\; 1\le i\le m,\; m\in \N\Big\}.\]
Let $T$ be a transversal of $(\centre_n(B),{+})$ in $(B,{+})$. Consider $b\in B$. Then there exists $u\in T$ such that $b\in u+\centre_n(B)=u\centre_n(B)$ so $b=uz$ for some $z\in \centre_n(B)$. Let $x\in B^n$. Then 
\[b\ast x=(uz)\ast x=u\ast (z\ast x)+z\ast x+u\ast x.\]
By Proposition~\ref{prop-cent-left-ser-star-0}, $z\ast x\in \centre_n(B)\ast B^n=0$. Therefore, $b\ast x=u\ast x$. Hence,
\begin{align*}
  B^{n+1}&=\Bigl\{\sum_{1\leq i\leq m}u_i\ast x_i\;\Bigm|\; u_i\in T,\; x_i\in B^n,\; 1\le i\le m,\; m\in\N\Bigr\}\\
         &=\Bigl\{\sum_{u\in T}u\ast y_u\;\Bigm|\; y_u\in B^n,\; u\in T,\; m\in\N\Bigr\}.
\end{align*}
The last equality follows, since $T$ is finite and $u\ast x+u\ast y=u\ast(x+y)$ for all $u\in T$ and $x$, $y\in B^n$.

Since  $|(B/{\centre(B)})/{\centre_{n-1}(B/{\centre(B)})}|=|B/{\centre_n(B)}|=t$, by induction, it follows that $(B^{n}+{\centre(B)})/{\centre(B)}=(B/{\centre(B)})^n$ is finite and has order at most ${\omega(t^2)}^{t^{n-2}}$. 
Hence, $B^n/(B^n\cap \centre(B))$ is finite and has order at most  ${\omega(t^2)}^{t^{n-2}}$. 

Given $u\in T$, consider $l_u\colon B^n\longrightarrow B^n$ given by $l_u(x)=u\ast x$. Since $(B, {+})$ is abelian, $l_u$ is a group endomorphism of $(B^n,{+})$ and we have that $B^n\cap \centre(B)\subseteq \ker l_u$ and
$l_u(B^n)\cong B^n/{\ker l_u}$. Since $B^n/(B^n\cap \centre(B))$ is finite and has order at most ${\omega(t^2)}^{t^{n-2}}$, $l_u(B^n)$ also is finite and has order at most ${\omega(t^2)}^{t^{n-2}}$. Since
\[B^{n+1}=\Big\{\sum_{u\in T}l_u(x_u)\;\Big|\; x_u\in B^n\Big\}\]
and $T$ has order $t$, it follows that, $B^{n+1}$ is finite and
\[|B^{n+1}|\leq \prod_{u\in T}|l_u(B^n)|\le {\Bigl({\omega(t^2)}^{t^{n-2}}\Bigr)}^t
  = {\omega(t^2)}^{t^{n-1}}.
  \qedhere\]

\end{proof}

\begin{prop}\label{prop-Gamma-left-serie}
If $B$ is a skew brace such that $(B, {+})$ is abelian and  $[B,B]^B$ is contained in $\ker\lambda$, then $\Gamma_n(B)=B^n$.
\end{prop}
\begin{proof}
For $n=1$, $\Gamma_1(B)=B=B^1$. Suppose that $n=2$. Then, since $(B, {+})$ is abelian,
  \[\Gamma_2(B)=[B,B]^B=\langle B*B+B*B+[B,B]_+\rangle^B=\langle B*B\rangle^B=\langle B^2\rangle^B.\]
  Moreover, $B^2$ is an ideal of $B$, and so $\Gamma_2(B)=B^2$ (in fact, the hypothesis $[B, B]^B\subseteq \ker\lambda$ has not been used for $n\le 2$). Hence we can suppose that $n\ge 3$ and that the result is true for $n-1$.
By \cite[Theorem~3.3]{BallesterEstebanFerraraPerezCTrombetti25-pacificjm-cent-nilp},
\[\Gamma_n(B)=[\Gamma_{n-1}(B), B]^B=\langle \Gamma_{n-1}(B)\ast B+B\ast \Gamma_{n-1}(B)+[\Gamma_{n-1}(B),B]_+\rangle^B.\]
We have that $\Gamma_{n-1}(B)\subseteq [B,B]^B\subseteq\ker\lambda$. Furthermore, $(B,{+})$ is abelian, hence, $\Gamma_n(B)=\langle B\ast \Gamma_{n-1}(B)\rangle^B=\langle B\ast B^{n-1}\rangle^B=\langle B^n\rangle^B$. We know that $B^n$ is a left ideal of $B$. Since $(B,{+})$ is abelian, $B^n$ is normal in $(B, {+})$. Moreover, $B^n\subseteq [B,B]^B\subseteq \ker\lambda$, hence, $B^n\ast B=0\subseteq B^n$. In conclusion, $B^n$ is an ideal of $B$. Therefore, $\Gamma_n(B)=B^n$. \qedhere
\end{proof}

\begin{proof}[Proof of Theorem~\ref{teo-baer-bound}]
  The result follows directly from Theorem~\ref{teo-left-ser-baer-ab-typ} and Proposition~\ref{prop-Gamma-left-serie}.\qedhere
\end{proof}

\section*{Acknowledgements}
This work has been supported by the grant PID2024-159495NB-I00, funded by MICIU/AEI/10.13039/501100011033 and by ERDF/EU. The first, the second and the fourth authors are supported by the grant CIAICO/2023/007 from the
Conselleria d’Educaci\'o, Universitats i Ocupaci\'o, Generalitat Valenciana. The
third author is very grateful to the Conselleria d’Innovació, Universitats,
Ciència i Societat Digital of the Generalitat (Valencian Community, Spain)
and the Universitat de València for their financial support to host researchers
affected by the war in Ukraine in research centres of the Valencian Community. The fourth author has been funded by the Spanish Ministry of Science, Innovation, and Universities through the FPU predoctoral grant FPU24/01319.

We thank sincerely the anonymous referee for reading carefully the manuscript and for making suggestions that have improved the presentation of our results.

\end{document}